\documentclass{amsart}
  
\newcommand{\Q}{\mathbb{Q}}
\newcommand{\C}{\mathbb{C}}

\newcommand{\N}{\mathbb{N}}

\newcommand{\cat}{^\frown}
\newcommand{\dom}{\operatorname{dom}}
\newcommand{\ran}{\operatorname{ran}}

\newcommand{\norm}[1]{\left\| #1 \right\|}
\newcommand{\supp}{\operatorname{supp}}

\newcommand{\B}{\mathcal{B}}
\newcommand{\zerovec}{\mathbf{0}}

\newcommand{\Fin}{\operatorname{Fin}}

\theoremstyle{theorem}
\newtheorem{theorem}{Theorem}[section]
\newtheorem{lemma}[theorem]{Lemma}

\usepackage{cite}

\theoremstyle{definition}
\newtheorem{definition}[theorem]{Definition}
\newtheorem{question}[theorem]{Question}

\theoremstyle{theorem}
\newtheorem{corollary}[theorem]{Corollary}

\theoremstyle{theorem}
\newtheorem{proposition}[theorem]{Proposition}

\theoremstyle{theorem}

\theoremstyle{theorem}

\theoremstyle{definition}

\theoremstyle{theorem}

\numberwithin{equation}{section}

\begin{document}
\title{Analytic computable structure theory and $L^p$-spaces part 2}
\author{Tyler A. Brown}
\address{Department of Mathematics\\
Iowa State University\\
Ames, Iowa 50011}
\email{tab5357@iastate.edu}
\author{Timothy H. McNicholl}
\address{Department of Mathematics\\
Iowa State University\\
Ames, Iowa 50011}
\email{mcnichol@iastate.edu}

\begin{abstract}
Suppose $p \geq 1$ is a computable real.  
We extend previous work of Clanin, Stull, and McNicholl by determining the degrees of categoricity of the separable $L^p$ spaces whose underlying measure spaces are atomic but not purely atomic.  In addition, we ascertain the complexity of associated projection maps.
\end{abstract}
\thanks{The second author was supported in by Simons Foundation Grant \# 317870.}
\maketitle

\section{Introduction}\label{sec:intro}

We continue here the program, recently initiated by Melnikov and Nies (see \cite{Melnikov.Nies.2013}, \cite{Melnikov.2013}), of utilizing the tools of computable analysis to investigate the effective structure theory of metric structures, in particular $L^p$ spaces where $p \geq 1$ is computable.  Specifically, we seek to classify the $L^p$ spaces that are computably categorical in that they have exactly one computable presentation up to computable isometric isomorphism.  We also seek to determine the degrees of categoricity of those $L^p$ spaces that are not computably categorical; this is the least powerful Turing degree that computes an isometric isomorphism between any two computable presentations of the space.

Recall that when $\Omega$ is a measure space, an \emph{atom} of $\Omega$ is a non-null measurable set $A$ so that $\mu(B) = \mu(A)$ for every non-null measurable set $B \subseteq A$.  A measure space with no atoms is \emph{non-atomic}, and a measure space is \emph{purely atomic} if its $\sigma$-algebra is generated by its atoms.  Recall also that with every measure space $\Omega$ there is an associated 
pseudo-metric $D_\Omega$ on its finitely measurable sets.  Namely, $D_\Omega(A,B)$ is the measure of the symmetric difference of $A$ and $B$.  The space $\Omega$ is said to be \emph{separable} if 
this associated pseudo-metric space is separable.  By means of convergence in measure, it is possible to show that an $L^p$ space is separable if and only if its underlying measure space is separable.

Suppose $p \geq 1$ is computable.  It is essentially shown in \cite{Pour-El.Richards.1989} that every separable $L^2$ space is computably categorical.  In \cite{McNicholl.2015}, the second author showed that $\ell^p$ is computably categorical only when $p \neq 2$.  Moreover, in \cite{McNicholl.2017} he showed that 
$\ell^p_n$ is computably categorical and that the degree of categoricity of $\ell^p$ is \textbf{0}''.  
Together, these results determine the degrees of categoricity of separable spaces of the form $L^p(\Omega)$ when $\Omega$ is purely atomic.  In the paper preceding this, Clanin, McNicholl, and Stull showed that $L^p(\Omega)$ is computably categorical when $\Omega$ is separable and nonatomic \cite{Clanin.McNicholl.Stull.2019}.    
Here, we complete the picture by determining the degrees of categoricity of separable $L^p$ spaces whose underlying measure spaces are atomic but not purely atomic.  Specifically, we show the following.

\begin{theorem}\label{thm:main}
Suppose $\Omega$ is a separable measure space that is atomic but not purely atomic, and suppose $p$ is a computable real so that $p \geq 1$ and $p \neq 2$.  Assume $L^p(\Omega)$ is nonzero.
\begin{enumerate}
	\item If $\Omega$ has finitely many atoms, then the degree of categoricity of $L^p(\Omega)$ is $\mathbf{0'}$.  \label{thm:main::itm:finite}
	
	\item If $\Omega$ has infinitely many atoms, then the degree of categoricity of $L^p(\Omega)$ is $\mathbf{0''}$.  \label{thm:main::itm:infinite}
\end{enumerate} 
\end{theorem}

Suppose $\Omega$ is a separable measure space that is not purely atomic, and suppose $L^p(\Omega)$ is nonzero where $1 \leq p < \infty$.  It follows from the Carath\'eodory classification of separable measure spaces that $L^p(\Omega)$ isometrically isomorphic to $L^p[0,1]$ if $\Omega$ has no atoms and that $L^p(\Omega)$ is isometrically isomorphic to $\ell^p_n \oplus L^p[0,1]$ if $\Omega$ has $n \geq 1$ atoms (see e.g. \cite{Cembranos.Mendoza.1997}).   It also follows that if 
$\Omega$ has infinitely many atoms, then $L^p(\Omega)$ is isometrically isomorphic to $\ell^p \oplus L^p[0,1]$.  Degrees of categoricity are preserved by isometric isomorphism.  We thus have the following.

\begin{corollary}\label{cor:lpnLp01}
Suppose $p$ is a computable real so that $p \geq 1$ and $p \neq 2$.  Then, the degree of categoricity of $\ell^p_n \oplus L^p[0,1]$ is $\mathbf{0'}$, and the degree of categoricity of $\ell^p \oplus L^p[0,1]$ is $\mathbf{0''}$.  
\end{corollary}

These results are somewhat surprising in that one might suspect that the spaces $\ell^p_n\oplus L^p[0,1]$ and $\ell^p\oplus L^p[0,1]$ do not have structure much different from their constituent (summand) spaces.  It turns out that allowing these summand spaces to ``work together" indeed produces a few complications in terms of their hybridized structure, particularly while establishing lower bounds for degrees of computable categoricity.  For example, in Section \ref{sec:lower}, we construct a computable presentation of $\ell^p_n\oplus L^p[0,1]$ so that projections of vectors into each of the summand spaces are incomputable.  However, as we will demonstrate in Section \ref{sec:upper}, it can be fruitful to dissect the hybrid case and consider the constituent spaces in tandem.  In fact, our results regarding the upper bounds for the degree of categoricity of $\ell^p_n\oplus L^p[0,1]$ and $\ell^p\oplus L^p[0,1]$ can be considered in this manner and largely piggyback on the results in \cite{Clanin.McNicholl.Stull.2019} and \cite{McNicholl.2017}.

Another consequence of these findings is that when $p \geq 1$ is computable, every separable $L^p$ space is $\mathbf{0''}$-categorical.

The paper is organized as follows.  Background and preliminaries are covered in Sections \ref{sec:back} and \ref{sec:prelim}.  In Section \ref{sec:projection}, we present results on the complexity of the natural projection maps for spaces of the form $\ell^p_n \oplus L^p[0,1]$ or $\ell^p \oplus L^p[0,1]$.  We derive lower bounds on degrees of categoricity in Section \ref{sec:lower} and corresponding upper bounds in Section \ref{sec:upper}.  These proofs utilize our results on projection maps in Section \ref{sec:projection}.  Finally in Section \ref{sec:conclusion} we summarize our findings and pose questions for further investigation.

\section{Background}\label{sec:back}
Here, we cover pertinent notions regarding external and internal direct sums of Banach spaces and the notion of complemented subspaces of an internal direct sum of Banach spaces.  We also summarize additional background material from \cite{Clanin.McNicholl.Stull.2019}.  We will assume our field of scalars consists of the complex numbers, but all results hold for the field of real numbers as well.

When $S \subseteq \N^*$, let $S\downarrow$ denote the downset of $S$; i.e. 
the set of all $\nu \in \N^*$ so that $\nu \subseteq \mu$ for some $\mu \in S$.  

Suppose $1 \leq p < \infty$.  If $\mathcal{B}_0$, $\ldots$, $\mathcal{B}_n$ are Banach spaces, their \emph{$L^p$-sum} consists of the vector space $\mathcal{B}_0 \times \ldots \times \mathcal{B}_n$ together with the norm
\[
\norm{(u_0, \ldots, u_n)}_p = \left( \sum_{j = 0}^n \norm{u_j}_{\mathcal{B}_j}^p \right)^{1/p}.
\]
Thus, the external direct sum of two $L^p$ spaces is their $L^p$-sum.
If $\mathcal{B}_j$ is a Banach space for each $j \in \N$, then 
the $L^p$-sum of $\{\mathcal{B}_j\}_{j \in \N}$ consists of all $f \in \prod_j B_j$ so that 
$\sum_j \norm{f(j)}_{\mathcal{B}_j}^p < \infty$.  This is easily seen to be a Banach space under the norm 
\[
\norm{f}_p = \left( \sum_j \norm{f(j)}_{\mathcal{B}_j}^p \right)^{1/p}.
\]

Suppose $\mathcal{B}$ is a Banach space and $\mathcal{M}$ and $\mathcal{N}$ are subspaces of $\mathcal{B}$.  
Recall that $\mathcal{B}$ is the \emph{internal direct sum} of $\mathcal{M}$ and $\mathcal{N}$ if $\mathcal{M} \cap \mathcal{N} = \{\zerovec\}$ and $\mathcal{B} = \mathcal{M} + \mathcal{N}$.  In this case, $\mathcal{M}$ is said to be \emph{complemented} and $\mathcal{N}$ is said to be the \emph{complement} of $\mathcal{M}$.
  When $\mathcal{M}$ is a complemented subspace of $\mathcal{B}$, let $P_{\mathcal{M}}$ denote the associated projection map. 
   That is, $P_{\mathcal{M}}$ is the unique linear map of $\mathcal{B}$ onto $\mathcal{M}$ so that $P_{\mathcal{M}}(f) = f$ for all $f \in \mathcal{M}$ and $P_{\mathcal{M}}(f) = \zerovec$ for all $f \in \mathcal{N}$.

Note that if $T$ is an isometric isomorphism of $\mathcal{B}_0$ onto $\mathcal{B}_1$, and if $\mathcal{M}$ is a complemented subspace of $\mathcal{B}_0$, then $T[\mathcal{M}]$ is a complemented subspace of 
$\mathcal{B}_1$ and $P_{T[\mathcal{M}]} = TP_{\mathcal{M}}T^{-1}$.

Suppose $f,g$ are vectors in an $L^p$ space.   We say that $f$ and $g$ are \emph{disjointly supported} if the intersection of their supports is null; equivalently, if $f \cdot g = \zerovec$.  
We say that $f$ is a \emph{subvector} of $g$ if there is a measurable set $A$ so that 
$f = g \cdot \chi_A$ (where $\chi_A$ is the characteristic function of $A$); equivalently, if 
$g - f$ and $f$ are disjointly supported.  We write $f \preceq g$ if $f$ is a subvector of $g$.  
It is readily seen that $\preceq$ is a partial order.  Also, $f$ is an atom of $\preceq$ if and only if 
$\supp(f)$ is an atom of the underlying measure space.

A subset $X$ of a Banach space $\B$ is \emph{linearly dense} if 
its linear span is dense in $\B$.

Suppose $S \subseteq \N^*$ is a tree and $\phi : S \rightarrow L^p(\Omega)$.  
We say that $\phi$ is \emph{summative} if for every nonterminal node $\nu$ of $S$, 
$\phi(\nu) = \sum_{\nu'} \phi(\nu')$ where $\nu'$ ranges over the children of $\nu$ in $S$.  
We say that $\phi$ is \emph{separating} if $\phi(\nu)$ and $\phi(\nu')$ are disjointly supported whenever $\nu, \nu' \in S$ are incomparable.  We then say that $\phi$ is a \emph{disintegration} if its range is linearly dense, and if it is injective, non-vanishing, summative, and separating.

Fix a disintegration $\phi : S \rightarrow L^p(\Omega)$.  
A non-root node $\nu$ of $S$ is an \emph{almost norm-maximizing} child of its parent if 
\[
\norm{\phi(\nu')}_p^p \leq \norm{\phi(\nu)}_p^p + 2^{-|\nu|}
\]
whenever $\nu' \in S$ is a sibling of $\nu$.  A chain $C \subseteq S$ is \emph{almost norm-maximizing} if 
for every $\nu \in C$, if $\nu$ has a child in $S$, then $C$ contains an almost norm-maximizing child of $\nu$.

Suppose $\B$ is a Banach space.  A \emph{structure} on $\B$ is a surjection of the natural numbers 
onto a linearly dense subset of $\B$.  A \emph{presentation of $\B$} is a pair $(\B, R)$ where $R$ is a structure on $\B$.  

Among all presentations of a Banach space $\mathcal{B}$, one may 
be designated as \emph{standard}; in this case, we will identify $\mathcal{B}$ with its standard presentation.
In particular, if $p \geq 1$ is a computable real, and if $D$ is a standard map of $\N$ onto the set of 
characteristic functions of dyadic subintervals of $[0,1]$, then $(L^p[0,1], D)$ is the standard presentation of $L^p[0,1]$.  If $R(n) = 1$ for all $n \in \N$, then $(\C, R)$ is the standard presentation of $\C$ as a Banach space over itself. The standard presentations of $\ell^p$ and $\ell^p_n$ are given by the standard
bases for these spaces.  The standard presentations of $\ell^p \oplus L^p[0,1]$ and $\ell^p_n \oplus L^p[0,1]$ are defined in the obvious way.

Fix a presentation $\B^\# = (\B, R)$ of a Banach space $\B$.  By a \emph{rational vector} of $\B^\#$ we mean a vector of the form $\sum_{j \leq M} \alpha_j R(j)$ where $\alpha_0, \ldots, \alpha_M \in \Q(i)$.  
We say $\B^\#$ is \emph{computable} if the norm function is computable on the rational vectors of $\B^\#$.  That is, if there is an algorithm that given $\alpha_0, \ldots, \alpha_M \in \Q(i)$ and $k \in \N$ produces a rational number $q$ so that $|\norm{\sum_{j \leq M} \alpha_j R(j)} - q| < 2^{-k}$.  The standard definitions just described are all easily seen to be computable.  A Banach space is \emph{computably presentable} if it has a computable presentation.

With a presentation $\B^\#$ of a Banach space, there are associated classes of computable vectors and sequences.  With a pair $(\B_0^\#, \B_1^\#)$ of presentations of Banach spaces, there is an associated class of computable functions from $\B_0^\#$ into $\B_1^\#$.  
Definitions of concepts such as these have become fairly well-known; we refer the reader to whom they are unfamiliar to Section 2.2.2 of \cite{Clanin.McNicholl.Stull.2019}.  

We will make frequent use of the following result from \cite{Clanin.McNicholl.Stull.2019}.  

\begin{theorem}\label{thm:disint.comp}
Suppose $p \geq 1$ is a computable real so that $p \neq 2$.  
Then, every computable presentation of a nonzero $L^p$ space has a computable disintegration. 
\end{theorem} 

The following is essentially proven in \cite{McNicholl.2017}.

\begin{theorem}\label{thm:anm.chains}
Suppose $p \geq 1$ is a computable real, and suppose $\B^\#$ is a computable presentation of an $L^p$ space.   If $\phi$ is a computable disintegration of $\B^\#$, then 
there is a partition $\{C_n\}_{n < \kappa}$ (where $\kappa \leq \omega$) of 
$\dom(\phi)$ into uniformly c.e. almost norm-maximizing chains.  
\end{theorem}

Degrees of categoricity for countable structures were introduced in \cite{Fokina.Kalimullin.Miller.2010}.  
Since then, the study of these degrees has given rise to a number of surprising results and very challenging questions; see, for example, \cite{Csima.Franklin.Shore.2013} and \cite{Anderson.Csima.2016}.  Any notion of classical computable structure theory can be adapted to the setting of Banach spaces by replacing `isomorphism' with `isometric isomorphism'.  Thus, we arrive at the previously given definition of the degree of categoricity for a computably presentable Banach space.

\section{Preliminaries}\label{sec:prelim}

\subsection{Preliminaries from functional analysis}\label{sec:prelim::subsec:FA}
Here we establish several preliminary lemmas and theorems from classical functional analysis that will be used later to prove Theorem \ref{thm:main}.  We first establish the results needed to locate the $\preceq$-atoms of $L^p(\Omega)$ via the use of almost norm-maximizing chains.  We then conclude this section with results regarding disintegrations on complemented subspaces of $L^p(\Omega)$.

The proof of the following is essentially the same as the proof of Proposition 4.1 of \cite{McNicholl.2017}.

\begin{proposition}\label{prop:subvectorLimitsExist}If $g_0 \preceq g_1 \preceq ...$ are vectors in $L^p(\Omega)$, then $\lim_n g_n$ exists in the $L^p$-norm and is the $\preceq$-infimum of $\{g_0,g_1,...\}$.
\end{proposition}

The following generalizes Theorem 3.4 of \cite{McNicholl.2017}.  

\begin{theorem}\label{thm:limitsAreAtoms}
Suppose $\Omega$ is a measure space and $\phi: S \rightarrow L^p(\Omega)$ is a disintegration.\\
\begin{enumerate}
	\item If $C \subseteq S$ is an almost norm-maximizing chain, then the $\preceq$-infimum of $\phi[C]$ exists and is either \textbf{0} or an atom of $\preceq$.  Furthermore, $\inf\phi[C]$ is the limit in the $L^p$ norm of $\phi(\nu)$ as $\nu$ traverses the nodes in $C$ in increasing order.\label{thm:limitsAreAtoms::itm:inf}
		
	\item If $\{C_n\}_{n=0}^\infty$ is a partition of $S$ into almost norm-maximizing chains, then $\inf\phi[C_0], \inf\phi[C_1], ...$ are disjointly supported.  Furthermore, if $A$ is an atom of $\Omega$, then there exists a unique $n$ so that $A$ is the support of $\inf \phi[C_n]$. \label{thm:limitsAreAtoms::itm:unique} 
\end{enumerate}
\end{theorem}

\begin{proof}
(\ref{thm:limitsAreAtoms::itm:inf}):  Suppose $C \subseteq S$ is an almost norm-maximizing chain. By Proposition \ref{prop:subvectorLimitsExist}, $g:=\inf \phi[C]$ exists and is the limit in the $L^p$-norm of $\phi(\nu)$ as $\nu$ traverses the nodes in $C$ in increasing order.  

We claim that $g$ is an atom if it is nonzero.  For, suppose $h \preceq g$.  Let $\delta = \min\{\norm{g - h}_p^p, \norm{h}_p^p\}$, and let 
$\epsilon> 0$.  Since the range of $\phi$ is linearly dense, there is a finite
$S_1 \subseteq S$ and a family of scalars $\{\alpha_\nu\}_{\nu \in S_1}$ so that 
\[
\norm{\sum_{\nu \in S_1} \alpha_\nu \phi(\nu) - h}_p < \frac{\epsilon}{2}.
\]
Let $f = \sum_{\nu \in S_1} \alpha_\nu \phi(\nu)$.  Then, 
\begin{eqnarray*}
\norm{f - g}_p^p & \geq & \norm{(f - h) \cdot \chi_{\supp(g)}}_p^p \\
& = & \norm{f \cdot \chi_{\supp(g)} - h}_p^p.
\end{eqnarray*}
Let $S_1^0 = \{\nu \in S_1\ :\ g \preceq \phi(\nu)\}$, and let $\beta = \sum_{\nu \in S_1^0} \alpha_\nu$.  Then, since $\phi$ is separating, 
$f \cdot \chi_{\supp(g)} = \beta g$.  However, since $g - h$ and $h$ are disjointly supported,
\begin{eqnarray*}
\norm{\beta g - h}_p^p & = & |\beta|^p\norm{g - h}_p^p + |\beta - 1|^p \norm{h}_p^p \\
& \geq & (|\beta|^p + |\beta -1|^p) \delta \\
& \geq & (|\beta|^p + ||\beta| - 1|^p) \delta\\
& \geq & \max\{|\beta|^p, ||\beta| - 1|^p\} \delta\\
& \geq & 2^{-p} \delta.
\end{eqnarray*}
Thus, $\delta < \epsilon$ for every $\epsilon > 0$.  Therefore, $\delta = 0$ and so either $g = h$ or $h = \zerovec$.  Thus, $g$ is an atom.
\\
	
	(\ref{thm:limitsAreAtoms::itm:unique}):   Suppose $C_0,C_1,...$ is a partition of $S$ into almost norm-maximizing chains.  By the above, $\inf \phi([C_k])$ exists for each $k$, and so we set $h_k := \inf \phi[C_k]$. We first claim that $h_0, h_1, \ldots$ are disjointly supported vectors.  Supposing that $k \neq k'$ it suffices to prove that there are incomparable nodes $\nu_0,\nu_1$ such that $\nu_0 \in C_k$ and $\nu_1 \in C_{k'}$.  We do this in two cases.
	
First, suppose there exist $\nu \in C_k, \nu' \in C_{k'}$ such that $|\nu| = |\nu'|$.
Since the chains $C_0,C_1,...$ partition $S$, $\nu \neq \nu'$.  Thus, $\nu_0:=\nu$ and $\nu_1:=\nu'$ are incomparable.
	
Now suppose $|\nu| \neq |\nu'|$ whenever $\nu \in C_k$ and $\nu' \in C_{k'}$.
	Let $\nu$ be the $\subseteq$-minimal node in $C_k$ and let $\nu'$ be the $\subseteq$-minimal node in $C_{k'}$.  Without loss of generality, assume $|\nu|<|\nu'|$.  Then $C_k$ must contain a terminal node $\tau_k$ of $S$, and $|\tau_k|<|\nu'|$.  Let $\mu\in S$ be the ancestor of $\nu'$ such that $|\mu|=|\tau_k|$.  Note that $\mu \notin C_{k'}$ since $|\mu| < |\nu'|$.  Furthermore, $\mu \notin C_k$ either, for $\tau_k$ is terminal in $S$.  Therefore, since $|\mu|=|\tau_k|$, $\mu$ and $\tau_k$ are incomparable.  From this it follows that $\nu_0:=\nu$ and $\nu_1:=\nu'$ are incomparable.
	
	Now let $A$ be an atom of $\Omega$.  If there is a $\preceq$-atom $g$ in $\ran(\phi)$ whose support includes $A$ then there is nothing to show.  So suppose that there is no atom in $\ran(\phi)$ whose support includes $A$.
	
	We claim that for each $n \in \N$ there is a $\nu \in S$ so that $|\nu|=n$ and $A \subseteq \supp(\phi(\nu))$.  For, suppose otherwise.  Since $\phi$ is summative and separating, it follows that $A \not \subseteq \supp(\phi(\nu))$ for all $\nu \in S$.  Let $\mu$ denote the measure function of $\Omega$.  Then, for any $g \in \ran(\phi)$, $\mu(A\cap\ \supp(g))=0$. Thus $\mu(A) \leq \norm{f -\chi_A}_p$ whenever $f$ belongs to the linear span of $\ran(\phi)$- a contradiction since the range of $\phi$ is linearly dense.  
	
	Now let $\nu_s$ denote the node of length $s$ so that $A\subseteq \supp(\phi(\nu_s))$.  Let $f = \phi(\emptyset) \cdot \chi_A$.  Then, $f \preceq \phi(\nu_s)$ for all $s$.  
	
	For each $s$, let $k_s$ denote the $k$ so that $\nu_s\in C_k$. We claim that $\lim_s k_s$ exists.  To see this, suppose otherwise.  Then we may let $s_0<s_1<...$ be the increasing enumeration of all values of $s$ so that $k_s\neq k_{s+1}$.  Since for all $m$, $\nu_{s_m+1} \supset \nu_{s_m}$, $\nu_{s_m}$ is a nonterminal node in $S$.  Thus since $C_{k_{s_m}}$ is almost norm-maximizing it must contain a child of $\nu_{s_m}$ in $S$; denote this child by $\mu_m$.  Then, $\phi(\mu_m)\preceq \phi(\nu_{s_m})$ and $\phi(\mu_m)$ and $\phi(\nu_{s_m+1})$ are disjointly supported.  Also, since $\mu_m$ is an almost norm-maximizing child of $\nu_{s_m}$, $\norm{\phi(\nu_{s_m+1})}^p_p \leq \norm{\phi(\mu_m)}_p^p+2^{-s_m}$.  Since $\phi(\mu_{m+r})\preceq \phi(\nu_{s_{m+r}})\preceq \phi(\nu_{s_m+1})$, $\phi(\mu_m)$ and $\phi(\mu_{m+r})$ are disjointly supported if $r>0$.  Thus by the above inequality and the summativity of $\phi$ we have
	\begin{align*}
		\sum_m\norm{\phi(\nu_{s_m+1})}_p^p
		&\leq \sum_m \norm{\phi(\mu_m)}_p^p + \sum_{m}2^{-s_m} \\
		&=\norm{\sum_m \phi(\mu_m)}_p^p + \sum_{m}2^{-s_m}\\
		&\leq \norm{\phi(\emptyset)}_p^p+ \sum_{m}2^{-s_m}\\
		&<\infty.
	\end{align*}
	But since $f \preceq \phi(\nu_{s_m+1})$ for all $m$, $\norm{\phi(\nu_{s_m + 1})}_p^p \geq \norm{f} > 0$ for all $m$- a contradiction.  
	
	Therefore, $k:=\lim_s k_s$ exists.  Since the chains partition $S$, $C_k$ is the only chain so that $A \subseteq \supp(\phi(\nu))$ for all $\phi(\nu)\in\phi[C_k]$.  It follows immediately from part (\ref{thm:limitsAreAtoms::itm:inf}) that $A$ is the support of $\inf \phi[C_k]$.  The result now follows.
\end{proof}

We say that subspaces $\mathcal{M}$, $\mathcal{N}$ of $L^p(\Omega)$ are disjointly supported if 
$f$, $g$ are disjointly supported whenever $f \in \mathcal{M}$ and $g \in \mathcal{N}$.  

\begin{lemma}\label{lm:proj.disint}
Suppose $\phi$ is a disintegration of $L^p(\Omega)$ and that $\mathcal{M}$ is a complemented subspace of $L^p(\Omega)$.  Suppose also that $\mathcal{M}$ and its complement are disjointly supported.  Then, $P_{\mathcal{M}} \phi$ is summative and separating, and its range is linearly dense in $\mathcal{M}$.
\end{lemma}

\begin{proof}
Let $P = P_{\mathcal{M}}$, and let $\psi = P\phi$.   Since $P$ is linear, it follows that $\psi$ is summative.  Since $\mathcal{M}$ and its complement are disjointly supported, it also follows that $\psi(\nu)$ is a subvector of $\phi(\nu)$ for each $\nu \in \dom(\phi)$.  We can then infer that $\psi$ is separating.  

We now show that the range of $P\phi$ is linearly dense in $\mathcal{M}$.  Let $\epsilon>0$.  By the linear density of $\phi$ and the disjointness of support of $\mathcal{M}$ and $\mathcal{N}$, for any $f\in\mathcal{M}$ there is a collection of scalars $\{\alpha_\nu\}_{\nu\in S}$ such that
\begin{align*}
\epsilon^p&>\norm{f-\sum_{\nu\in S}\alpha_\nu \phi(\nu)}_p^p\\
&=\norm{f-\sum_{\nu\in S}\alpha_\nu P(\phi(\nu))-\sum_{\nu\in S}\alpha_\nu P_{\mathcal{N}}(\phi(\nu))}_p^p\\
&=\norm{f-\sum_{\nu\in S}\alpha_\nu P(\phi(\nu))}^p_p+\norm{\sum_{\nu\in S}\alpha_\nu P_{\mathcal{N}}(\phi(\nu))}_p^p\\
&\geq \norm{f-\sum_{\nu\in S}\alpha_\nu P(\phi(\nu))}^p_p.
\end{align*}
Thus we have that the range of $P \phi$ is linearly dense in $\mathcal{M}$.
\end{proof}

\begin{lemma}\label{lm:proj.disint.Lp}
Suppose $\phi$ is a disintegration of $\mathcal{M} \oplus L^p[0,1]$ where $\mathcal{M}$ is either $\ell^p$ or $\ell^p_n$.  Suppose $\{C_n\}_{n \in \N}$ is a partition of $\dom(\phi)$ into almost norm-maximizing chains and that $g_n = \inf \phi[C_n]$ for all $n$.  Then, for each $\nu \in \dom(\phi)$, 
\[
P_{\{\zerovec\} \oplus L^p[0,1]} \phi(\nu) = \phi(\nu) - \sum_{g_n \preceq \phi(\nu)} g_n.
\]
\end{lemma}

\begin{proof}
Let $\mathcal{B} = \mathcal{M} \oplus L^p[0,1]$, and let $P = P_{\{\zerovec\} \oplus L^p[0,1]}$.  For each $f \in \mathcal{B}$, let $\mathcal{A}_f$ denote the set of all atoms $g$ of $\mathcal{B}$ so that $g \preceq f$.  Thus, $P(f) = f - \sum_{g \in \mathcal{A}_f} g$.  Suppose $g \in \mathcal{A}_{\phi(\nu)}$.  Then, $\supp(g)$ is an atom.  So, by Theorem \ref{thm:limitsAreAtoms}, $\supp(g) = \supp(g_n)$ for some $n$.   

We claim that $\nu \in C_n\downarrow$.  For, suppose $\nu \not \in C_n\downarrow$.  Let $\nu'$ be the largest node in $C_n\downarrow$ so that $\nu' \subseteq \nu$.  Thus, $\nu' \neq \nu$ so and $\nu'$ has a child in $\dom(\phi)$.  Therefore, $\nu'$ has a child $\nu''$ in $C_n$ since $C_n$ is almost norm-maximizing.  Thus, $\nu''$ and $\nu$ are incomparable.  It follows that $g_n$ and $g$ are disjointly supported- a contradiction.

Since $\nu \in C_n\downarrow$, it follows that $g,g_n \preceq \phi(\nu)$.  Thus, since $\supp(g) = \supp(g_n)$, $g = g_n$. 
\end{proof}

\subsection{Preliminaries from computable analysis}\label{sec:prelim::subsec:CA}
This subsection essentially effectivizes the notions of the previous subsection and makes explicit the computable presentations we will employ in the proofs of our main theorem and its corollary.

The following is from \cite{Pour-El.Richards.1989}.

\begin{theorem}\label{thm:seq.comp.map}
Suppose $\mathcal{B}_0^\#$ and $\mathcal{B}_1^\#$ are computable presentations of Banach spaces 
$\mathcal{B}_0$ and $\mathcal{B}_1$ respectively and that $T : \mathcal{B}_0^\# \rightarrow \mathcal{B}_1^\#$ is bounded and linear.  Then, $T$ is computable if and only if $T$ maps a linearly dense computable sequence of $\mathcal{B}_0^\#$ to a computable sequence of $\mathcal{B}_1^\#$.   
\end{theorem}

\begin{definition}\label{def:compu.comp}
Suppose $\mathcal{B}^\#$ is a computable presentation of a Banach space $\mathcal{B}$, and suppose 
$\mathcal{M}$ is a complemented subspace of $\mathcal{B}$.  We say $\mathcal{M}$ is a 
\emph{computably complemented subspace of $\mathcal{B}$} if $P_\mathcal{M}$ is a 
computable map of $\mathcal{B}^\#$ into $\mathcal{B}^\#$.
\end{definition}

We relativize this notion in the obvious way.

\begin{proposition}\label{prop:proj.comp}
Suppose $\mathcal{B}_j^\#$ is a computable presentation of a Banach space $\mathcal{B}_j$ for 
each $j \in \{0,1\}$, and suppose $T$ is an $X$-computable isometric isomorphism of 
$\mathcal{B}_0^\#$ onto $\mathcal{B}_1^\#$.  If $\mathcal{M}$ is a computably complemented subspace of $\mathcal{B}_0^\#$, then $T[\mathcal{M}]$ is an  $X$-computably complemented subspace of $\mathcal{B}_1^\#$.
\end{proposition}

\begin{proof}
This is clear from the fact that $P_{T[\mathcal{M}]} = TP_{\mathcal{M}}T^{-1}$.  
\end{proof}

\begin{lemma}\label{lm:disintsArePresentations}
Let $p \geq 1$ be computable.
Suppose $S$ is a tree, and suppose $\phi : S \rightarrow L^p(\Omega)$ is summative and separating.  Suppose also that $\ran(S)$ is linearly dense and that $\nu \mapsto \norm{\phi(\nu)}_p$ is computable.
	Let $R = \phi h$ where $h$ is a computable surjection of $\N$ onto $S$.  Then, $(L^p(\Omega), R)$ is a computable presentation of $L^p(\Omega)$.
\end{lemma}

\begin{proof}
 Since $\ran(\phi)$ is linearly dense, it follows that $R$ is a structure on $L^p(\Omega)$ and that $L^p(\Omega)^\# := (L^p(\Omega),R)$ is a presentation of $L^p(\Omega)$.  
 
 Now we must demonstrate that this presentation is computable.  That is, we must show that the norm function is computable on the rational vectors of $L^p(\Omega)^\#$.  So, suppose $\alpha_0, \ldots, \alpha_M \in \Q(i)$ are given, and let $f = \sum_j \alpha_j R(j)$.  Compute a finite tree $F \subseteq S$ so that $R(j) \in F$ for each $j \leq M$.  
For each $\nu \in F$, let $\alpha_\nu = \sum_{h(j) = \nu} \alpha_j$.  Thus, 
$\sum_j \alpha_j R(j) = \sum_\nu \alpha_\nu \phi(\nu)$.  Let $\beta_0, \ldots, \beta_k$ denote the leaf nodes of $F$.  Thus, $\supp(f) = \bigcup_j \supp(\phi(\beta_j))$.   Therefore, 
\begin{eqnarray*}
\norm{f}_p^p & = & \sum_j \norm{f \cdot \chi_{\supp(\phi(\beta_j))}}_p^p\\
& = & \sum_j \norm{ \left(\sum_{\nu \subseteq \beta_j} \alpha_\nu\right) \phi(\beta_j)}_p^p\\
& = & \sum_j \left| \sum_{\nu \subseteq \beta_j} \alpha_\nu \right|^p \norm{\phi(\beta_j)}_p^p. 
\end{eqnarray*}
Since $\nu \mapsto \norm{\phi(\nu)}_p$ is computable, it follows that $\norm{f}_p$ can be computed from 
$\alpha_0, \ldots, \alpha_M$.
 \end{proof}

\section{Complexity of projection maps}\label{sec:projection}
Here we establish the complexity of projection maps on the spaces $\ell^p_n\oplus L^p[0,1]$ and $\ell^p\oplus L^p[0,1]$ respectively.  The main theorem of this section is the core of our argument yielding the upper bounds for each of the aforementioned spaces.

\begin{theorem}\label{thm:proj.Lp.comp}
Let $p \geq 1$ be a computable real besides $2$. 
Suppose $\mathcal{M}$ is either $\ell^p_n$ or $\ell^p$, and suppose 
$(\mathcal{M} \oplus L^p[0,1])^\#$ is a computable presentation of 
$\mathcal{M} \oplus L^p[0,1]$.  
\begin{enumerate}
	\item If $\mathcal{M} = \ell^p_n$, then $P_{\{\zerovec\} \oplus L^p[0,1]}$ is a 
	$\emptyset'$-computable map of $(\mathcal{M} \oplus L^p[0,1])^\#$ into 
	$(\mathcal{M} \oplus L^p[0,1])^\#$.
	
	\item If $\mathcal{M} = \ell^p$, then $P_{\{\zerovec\}\oplus L^p[0,1]}$ is a 
	$\emptyset''$-computable map of $(\mathcal{M} \oplus L^p[0,1])^\#$ into 
	$(\mathcal{M} \oplus L^p[0,1])^\#$.
\end{enumerate}
\end{theorem}

\begin{proof}
Let $\mathcal{B} = \mathcal{M} \oplus L^p[0,1]$, and let $\phi$ be a computable disintegration of 
$\mathcal{B}^\#$.  Set $S = \dom(\phi)$.  Abbreviate $P_{\{\zerovec\} \oplus L^p[0,1]}$ by $P$.

Let $\mathcal{B}^\# = (\mathcal{B}, R)$.  Fix a computable surjection $h$ of $\N$ onto $S$, and set $R'(j) = \phi(h(j))$.  Let $\mathcal{B}^+ = (\mathcal{B}, R')$.  By Lemma \ref{lm:disintsArePresentations}, $\mathcal{B}^+$ is 
a computable presentation of $\mathcal{B}$.  Furthermore, since $R'$ is a computable sequence of 
$\mathcal{B}^\#$, it follows from Theorem \ref{thm:seq.comp.map} that $\mathcal{B}^\#$ is computably isometrically isomorphic to 
$\mathcal{B}^+$ (namely, by the identity map).   

By Theorem \ref{thm:limitsAreAtoms}, there is a partition $\{C_j\}_{j \in \N}$ of $S$ into almost norm-maximizing chains.  
Let $g_j = \inf \phi[C_j]$.  By Lemma \ref{lm:proj.disint.Lp}, 
\[
P(\phi(\nu)) = \phi(\nu) - \sum_{g_j \preceq \phi(\nu)} g_j.
\]

Let $U_\nu = \{j \in \N\ :\ \nu \in C_j\downarrow\}$.  We first claim that 
\[
\sum_{g_j \preceq \phi(\nu)} g_j = \sum_{j \in U_\nu} g_j.
\]
For, if $j \in U_\nu$, then $g_j \preceq \phi(\nu)$.  Suppose $j \not \in U_\nu$ and $g_j \preceq \phi(\nu)$.  
Let $\mu_0$ be the maximal element of $C_n\downarrow$ so that $\mu_0 \subseteq \nu$.  Thus, 
$\mu_0 \subset \nu$, and so $\mu_0$ has a child $\mu'$ in $S$.  Therefore $g_j \preceq \phi(\mu'), \phi(\nu)$.  
Since $\phi$ is separating, it follows that $g_j = \zerovec$.  

Now, suppose $\mathcal{M} = \ell^p_n$.  We obtain from Theorem \ref{thm:limitsAreAtoms}, that there are exactly $n$ values of $j$ so that $g_j$ is nonzero.  So, let $D = \{j\ : g_j \neq \zerovec\}$.  
Then, $\phi(\nu) - P(\phi(\nu)) = \sum_{j \in U_\nu \cap D} g_j$.  It then follows from Theorem \ref{thm:limitsAreAtoms} that $\{g_j\}_{j \in \N}$ is a $\emptyset'$-computable sequence of $\mathcal{B}^\#$.  
Thus, $\{P(R'(j))\}_{j \in \N}$ is a 
$\emptyset'$-computable sequence of $\mathcal{B}^+$.  Therefore, by the relativization of Theorem \ref{thm:seq.comp.map}, 
$P$ is a $\emptyset'$-computable map of $\mathcal{B}^+$ into $\mathcal{B}^+$.  But, since 
$\mathcal{B}^\#$ is computably isometrically isomorphic to $\mathcal{B}^+$, $P$ is also a $\emptyset'$-computable map of $\mathcal{B}^\#$ into $\mathcal{B}^\#$.

Now, suppose $\mathcal{M} = \ell^p$.  For each $\nu \in S$, let $h_\nu = \phi(\nu) - P(\phi(\nu))$.  
Since $\{g_j\}_{j \in U_\nu}$ is a summable sequence of disjointly supported vectors, 
$\sum_{j \in U_\nu} \norm{g_j}_p^p < \infty$.  Moreover, since $\{g_j\}_{j \in U_\nu}$ is a $\emptyset'$-computable sequence of $\mathcal{B}^\#$, it follows that $\sum_{j \in U_\nu} \norm{g_j}_p^p$ is $\emptyset''$-computable uniformly in $\nu$.  Observe that for each $N \in \N$,
\begin{eqnarray*}
\norm{\sum_{j \in U_\nu \cap [0,N]} g_j - h_\nu}_p^p & = & \norm{\sum_{j \in U_\nu \cap [N+1, \infty)} g_j}_p^p\\
& = & \sum_{j \in U_\nu \cap [N+1, \infty)} \norm{g_j}_p^p.
\end{eqnarray*}  
From this we obtain that $h_\nu$ is a $\emptyset''$-computable vector of $\mathcal{B}^\#$ uniformly in $\nu$.  It then follows that $\{P(R'(j))\}_{j \in \N}$ is a $\emptyset''$-computable sequence of $\mathcal{B}^+$ and so $P$ is a $\emptyset''$-computable map of $\mathcal{B}^+$ into $\mathcal{B}^+$.
\end{proof}

The sharpness of the bounds in Proposition \ref{prop:proj.comp} will be demonstrated in Section \ref{sec:lower}.

\section{Upper bound results}\label{sec:upper}
Here we will use the complexity of projection maps described in the previous section to produce the upper bounds for the degree of categoricity of $\ell^p_n\oplus L^p[0,1]$ and $\ell^p\oplus L^p[0,1]$ respectively.

\begin{theorem}\label{thm:upper}
Suppose $p \geq 1$ is a computable real so that $p \neq 2$.  Then, $\ell^p_n \oplus L^p[0,1]$ is $\emptyset'$-categorical, and $\ell^p \oplus L^p[0,1]$ is $\emptyset''$-categorical.
\end{theorem}

\begin{proof}
Suppose $\mathcal{A}$ is either $\ell^p_n$ or $\ell^p$, and let $\mathcal{B} = \mathcal{A} \oplus L^p[0,1]$.  Let $\mathcal{B}^\#$ be a computable presentation of $\mathcal{B}$, and let $\phi$ be a computable disintegration of $\mathcal{B}^\#$.  Let $S = \dom(\phi)$.  Since $\mathcal{B}$ is infinite-dimensional, $S$ is infinite.  Fix a computable surjection $h$ of $\N$ onto $S$.  
Let $\mathcal{M} = \mathcal{A} \oplus \{\zerovec\}$, and let $\mathcal{N} = \{\zerovec\} \oplus L^p[0,1]$.  
In addition, let $P = P_{\mathcal{N}}$.  

We first claim that there is a $\emptyset'$-computable map $T_1 : \mathcal{A} \rightarrow \mathcal{B}^\#$ so that $\ran(T_1) = \mathcal{M}$.  For suppose $\mathcal{A} = \ell^p_n$.   Then, by Theorem \ref{thm:proj.Lp.comp}, $P$ is a $\emptyset'$-computable map of $\mathcal{B}^\#$ into $\mathcal{B}^\#$.  Let $\mathcal{M}^\# = (\mathcal{M}, (I-P)\phi h)$.  By the relativization of Lemma \ref{lm:disintsArePresentations}, $\mathcal{M}^\#$ is a $\emptyset'$-computable presentation of $\mathcal{M}$.  In Section 6 of \cite{McNicholl.2017}, it is shown that $\ell^p_n$ is computably categorical.  So, by relativizing this result, there is a $\emptyset'$-computable isometric isomorphism $T_1$ of $\ell^p_n$ onto $\mathcal{M}^\#$.  
Since $(I - P)\phi h$ is a $\emptyset'$-computable sequence of $\mathcal{B}^\#$, by the relativization of Theorem \ref{thm:seq.comp.map}, $T_1$ is a $\emptyset'$-computable map of $\ell^p_n$ into $\mathcal{B}^\#$.   

Now, suppose $\mathcal{A} = \ell^p$.  By Theorem \ref{thm:anm.chains}, there is a partition $\{C_n\}_{m < \kappa}$ of $S$ into uniformly c.e. almost norm-maximizing chains; since $S$ is infinite, it follows that $\kappa = \omega$.  Let $g_n = \inf \phi[C_n]$.  Then, there is a $\emptyset'$-computable one-to-one enumeration $\{n_k\}_{k = 0}^\infty$ of all $n$ so that $g_n$ is nonzero.  By Theorem \ref{thm:limitsAreAtoms}, for each $j \in \N$, there is a unique $k$ so that $\{j\} = \supp(g_{n_k})$.  Let $T_1$ be the unique linear map of $\ell^p$ into $\mathcal{N}$ so that $T_1(e_k) = \norm{g_{n_k}}_p^{-1} g_{n_k}$ for all $k$.  Since the $g_{n_k}$'s are disjointly supported, it follows that $T_1$ is isometric.   
It follows from the relativization of Theorem \ref{thm:seq.comp.map} that $T_1$ is a $\emptyset'$-computable map of $\ell^p$ into $\mathcal{B}^\#$.

We now claim that if $\mathcal{A} = \ell^p_n$, then there is a $\emptyset'$-computable map $T_2$ of $L^p[0,1]$ into $\mathcal{B}^\#$ so that $\ran(T_2) = \mathcal{N}$.  For, in this case, by Theorem \ref{thm:proj.Lp.comp}, $P$ is a $\emptyset'$-computable map of $\mathcal{B}^\#$ into $\mathcal{B}^\#$.  
Thus, by Lemma \ref{lm:proj.disint}, $\mathcal{N}^\# = (\mathcal{N}, P \phi h)$ is a $\emptyset'$-computable presentation of 
$\mathcal{N}$.  So by the relativization of Theorem 1.1 of \cite{Clanin.McNicholl.Stull.2019}, 
there is a $\emptyset'$-computable isometric isomorphism $T_2$ of $L^p[0,1]$ onto $\mathcal{N}^\#$.  
Since $P \phi h$ is a $\emptyset'$-computable sequence of $\mathcal{B}^\#$, 
$T_2$ is a $\emptyset'$-computable map of $L^p[0,1]$ into $\mathcal{B}^\#$ by the relativization of Theorem \ref{thm:seq.comp.map}.  

It similarly follows that when $\mathcal{A} = \ell^p$, there is a $\emptyset''$-computable map $T_2$ of $L^p[0,1]$ into $\mathcal{B}^\#$ so that $\ran(T_2) = \mathcal{N}$.  

We now form a map $T$ by gluing the maps $T_1$ and $T_2$ together.  Namely, when, $v \in \mathcal{A}$ and $f \in L^p[0,1]$, let $T(v,f) = T_1(v) + T_2(f)$.  Thus, $T$ is an isometric automorphism of $\B$.  If $\mathcal{A} = \ell^p_n$, then $T$ is a $\emptyset'$-computable map of the standard presentation of $\B$ onto $\B^\#$; otherwise it is a $\emptyset''$-computable map between these presentations.   
\end{proof}

\section{Lower bound results}\label{sec:lower}
In each of the following subsections we construct ill-behaved computable presentations for $\ell^p_n\oplus L^p[0,1]$ and $\ell^p\oplus L^p[0,1]$ respectively. We then show that any oracle that computes a linear isometric isomorphism between each constructed presentation and its standard copy must also compute \textbf{d}, where \textbf{d} is any c.e. degree in the $\ell^p_n\oplus L^p[0,1]$ case and $\mathbf{d}$ is the degree of $\emptyset''$ in the $\ell^p\oplus L^p[0,1]$ case.

\subsection{The finitely atomic case}\label{sec:lower::subsec:lpn}

We complete our proof of Theorem \ref{thm:main}.\ref{thm:main::itm:finite} by establishing the following.

\begin{theorem}\label{thm:lpn.lower}
Suppose $p \geq 1$ is computable and $p \neq 2$.  Let $\mathbf{d}$ be a c.e. degree.  Then, there is a computable presentation $(\ell^p_n \oplus L^p[0,1])^\#$ of $\ell^p_n \oplus L^p[0,1]$ so that every degree that computes an isometric isomorphism of $\ell^p \oplus L^p[0,1]$ onto $(\ell^p \oplus L^p[0,1])^\#$ also computes $\mathbf{d}$.
\end{theorem}

Let $\mathcal{B} = \ell^p_n \oplus L^p[0,1]$.  We construct $\mathcal{B}^\#$ as follows.  We first construct a disintegration $\phi$ of $\mathcal{B}$.  Let $\gamma\in(0,1)$ be a left-c.e. real so that the left Dedekind cut of $\gamma$ has Turing degree $\mathbf{d}$.  Let $\{q_n\}$ be a computable and increasing sequence of positive rational numbers so that $\lim_j q_j = \gamma$.  Let $c = 1 - \gamma + q_0$.  Define
\begin{eqnarray*}
a((1)) & = & 1 - c\\
b((1)) & = & 1\\
a({(0)^{j+1}} \cat(1)) & = & \gamma - q_j\\
b({(0)^{j+1}}\cat(1)) & = & \gamma - q_{j-1}.
\end{eqnarray*}
Assuming $a(\nu)$ and $b(\nu)$ have been defined, set $a(\nu\cat(0)) = a(\nu)$, $a(\nu\cat(1)) = b(\nu\cat(0)) = \frac{1}{2} (a(\nu) + b(\nu))$, and $b(\nu\cat(1)) = b(\nu)$.  

Now, let:
\begin{eqnarray*}
\phi(\emptyset) & = & ((1 - \gamma)^{1/p} e_0 + e_1 + \ldots + e_{n-1}, \chi_{[0, 1 - c]} + c^{-1/p}\chi_{[1-c,1]})\\
\phi((0)^{j+1}) & = & ((1 - \gamma)^{1/p}e_0, \chi_{[\gamma - q_j, 1]})\\
\phi((j)) & = & (e_{j-1}, \zerovec)\ \mbox{if $2 \leq j < n$}\\
\phi(\mu) & = & c^{-1/p}(\zerovec, \chi_{[a(\mu), b(\mu)]})\ \mbox{if $(1) \subseteq \mu$}\\
\phi(\mu) & = & (\zerovec, \chi_{[a(\mu), b(\mu)]})\ \mbox{if ${(0)^{j+1}}\cat(1) \subseteq \mu$}
\end{eqnarray*}

\begin{lemma}\label{lm:phi.disint}
$\phi$ is a disintegration of $\mathcal{B}$.
\end{lemma}

\begin{proof}
By construction, $\phi$ is summative, separating, injective, and never zero.  It only remains to show that $\ran(\phi)$ is linearly dense.  By construction, $(e_j, \zerovec) \in ran(\phi)$ when $1 \leq j \leq n - 1$.  So, it is enough to show that $(e_0,\zerovec) \in \langle \ran(\phi) \rangle$ and that $(0,\chi_I) \in \langle \ran(\phi)\rangle$ for every closed interval $I \subseteq [0,1]$.  
		
Let $\epsilon >0$ be given.  There is a $K\in \N$ so that $|\gamma-q_K|/|1 - \gamma| < \epsilon^p$.  Furthermore, 
	\begin{align*}
		\norm{(e_0, \textbf{0}) - \frac{1}{(1-\gamma)^{1/p}} \phi((0)^{K+1})}_p^p&=\norm{(\zerovec,-\frac{1}{(1-\gamma)^{1/p}}\chi_{[0,(\gamma - q_K)]})}_p^p\\
		&=\frac{|\gamma-q_K|}{|1-\gamma|}\\
		&<\epsilon^p
	\end{align*}
	Therefore $(e_0,\zerovec) \in \langle \ran(\phi)\rangle$.
	
Let $\mathcal{M} = \{\zerovec\} \oplus L^p[0,1]$, and let $E(f) = \supp(P_\mathcal{M}(f))$.  
By construction, for each $f \in \ran(\phi)$, $\chi_{E(f)}$ belongs to the linear span of $\ran(\phi)$.  By induction, 
\[
\bigcup_{|\nu| = j} E(\phi(\nu)) = [0,1].
\]	
Since $\phi$ is separating, $\{E(\phi(\nu))\}_{|\nu| = j}$ is a partition of $[0,1]$.  Let $L_\nu$ denote the 
length of $E(\phi(\nu))$.  Then, by construction, $\lim_j \max_{|\nu| = j} L_\nu = 0$.  
It follows that if $I \subseteq [0,1]$ is a closed interval, then $\chi_I$ belongs to the closed linear span of $\ran(\phi)$.  
\end{proof}

\begin{lemma}\label{lm:norm.phi.comp}
$\nu \mapsto \norm{\phi(\nu)}_p$ is computable.
\end{lemma}

\begin{proof}
We have:
\begin{eqnarray*}
\norm{\phi(\emptyset)}_p^p & = & n+1 - q_0\\
\norm{\phi((0)^{j+1})}_p^p&  = & 1 - q_j\\
\end{eqnarray*}
If $2 \leq j < n$, then $\norm{\phi((j))}_p^p = 1$.  If $(1) \subseteq \mu$, then 
$\norm{\phi(\mu)}_p^p = c^{-1} c (b(\mu) - a(\mu)) = 2^{-|\mu| + 1}$.  
Moreover, if ${(0)^{j+1}}\cat(1) \subseteq \mu$, then 
$\norm{\phi(\mu)}_p^p = b(\mu) - a(\mu) = 2^{-|\mu| + j + 2}(q_j - q_{j-1})$.  
Thus, $\nu \mapsto \norm{\phi(\nu)}_p$ is computable.
\end{proof}

Therefore, $\mathcal{B}^\#$ is a computable presentation of $\mathcal{B}$ by Lemma \ref{lm:disintsArePresentations}.

\begin{lemma}\label{lm:proj}
If the projection $P_{\langle e_0 \rangle \oplus \{\zerovec\}}$ is an $X$-computable map of 
$\mathcal{B}^\#$ into $\mathcal{B}^\#$, then 
$X$ computes $\mathbf{d}$.
\end{lemma}

\begin{proof}
Let $P = P_{\langle e_0 \rangle \oplus \{\zerovec\}}$.  Suppose $P$ is an $X$-computable map of 
$\mathcal{B}^\#$ into $\mathcal{B}^\#$.  Let $f = \phi((0))$.  Thus, $f$ is a computable vector of $\mathcal{B}^\#$, and so $X$ computes $(1 - \gamma)^{1/p} = \norm{P(f)}_p$.  Therefore, $X$ computes $\gamma$ and so $X$ computes $\mathbf{d}$.  
\end{proof}

Suppose $X$ computes an isometric isomorphism $T$ of $\mathcal{B}$ onto 
$\mathcal{B}^\#$.  Since $T$ preserves the subvector ordering, there is a $j_0$ so that $T((e_{j_0}, \zerovec))$ is a nonzero scalar multiple of $(e_0, \zerovec)$.  Let $\mathcal{M} = \langle (e_{j_0}, \zerovec) \rangle$.  
Then, $T[\mathcal{M}] = \langle (e_0, \zerovec) \rangle$.  Let $P = P_{T[\mathcal{M}]}$.  By Theorem \ref{thm:proj.Lp.comp}, $P$ is an $X$-computable map of $\mathcal{B}^\#$ into itself.  
Let $f = \phi((0))$.  Thus, $f$ is a computable vector of $\mathcal{B}^\#$.  Hence, 
$\norm{P(f)}_p = (1 - \gamma)^{1/p}$ is $X$-computable.  Therefore, $X$ computes $\gamma$, and so 
$X$ computes $\mathbf{d}$.

Note that we have also established the sharpness of the bound in Theorem \ref{thm:proj.Lp.comp}.

\subsection{The infinitely atomic case}\label{sec:lower::subsec:lp}

We complete our proof of Theorem \ref{thm:main} by proving the following.

\begin{theorem}\label{thm:lpLp01.lower}
Suppose $p \geq 1$ is computable and $p \neq 2$.  
There is a computable presentation $\mathcal{B}^\#$ of $\ell^p \oplus L^p[0,1]$ so that 
every oracle that computes an isometric isomorphism of $\ell^p \oplus L^p[0,1]$ onto $\mathcal{B}^\#$ also  computes $\emptyset''$.  
\end{theorem}

We construct $\mathcal{B}$ as follows.  Let 
\[
m_e = \left\{\begin{array}{ll}
\# W_e & \mbox{\ if $e \in \Fin$;}\\
\omega & \mbox{\ otherwise.}\\
\end{array}
\right.
\]
For each $e \in \N$, let 
\[
\mathcal{B}_e = \left\{ \begin{array}{ll}
\ell^p_{2^{m_e}} & \mbox{\ if $e \in \Fin$;}\\
L^p[0,1] & \mbox{\ otherwise.}\\
\end{array}
\right.
\]
Let $\mathcal{B}$ be the $L^p$ sum of $\{\mathcal{B}_e\}_{e \in \N}$.  
Let $\iota_e$ be the natural injection of $\mathcal{B}_e$ into $\mathcal{B}$.

We now build a presentation of $\mathcal{B}$ via the construction of a disintegration $\phi$ of $\mathcal{B}$.  Let
\[
S = \omega^{\leq 1}\ \cup\ \{(e)\cat\alpha\ :\ \alpha \in \{0,1\}^{< m_e}\}.
\]
Thus, $S$ is c.e..
Let 
\[
g_e = \left\{\begin{array}{ll}
2^{-m_e/p} \sum_{j < 2^{m_e}} e_j & \mbox{\ if $e \in \Fin$;}\\
\chi_{[0,1]} & \mbox{\ otherwise.}\\
\end{array}
\right.
\]
Let $f_e = \iota_e(g_e)$.  For each $e$ we let $\phi((e)) = 2^{-(e+1)}f_e$.  

For each $\nu \in S - \{\emptyset\}$, we recursively define a set $I(\nu)$ as follows.  
For each $e \in \N$, let 
\[
I((e)) = \left\{ \begin{array}{ll}
		\{0, \ldots, 2^{m_e} - 1\} & \mbox{\ if $e \in \Fin$;}\\ \relax
		[0,1] & \mbox{\ otherwise.}\\
		\end{array}
		\right.
\] 
Suppose $\nu \in S$ and $I(\nu)$ has been defined.  Let $a(\nu) = \min I(\nu)$, and let $b(\nu) = \max I(\nu)$.
Let $e = \nu(0)$.  If $e \not \in \Fin$, let: 
\begin{eqnarray*}
I(\nu\cat(0)) & = & [a(\nu), 2^{-1}(a(\nu) + b(\nu))]\\
I(\nu \cat(1)) & = & [2^{-1}(a(\nu) + b(\nu)), b(\nu)]
\end{eqnarray*}
If $e \in \Fin$, and if $|\nu| + 1 < m_e$, let:
\begin{eqnarray*}
I(\nu\cat(0)) & = & \{a(\nu), \ldots, a(\nu) + \frac{1}{2}\#I(\nu) - 1\}\\
I(\nu\cat(1)) & = & \{ \frac{1}{2}\#I(\nu), \ldots, b(\nu)\}
\end{eqnarray*}

When $\nu \in S$, let 
\[
\phi(\nu) = \left\{
\begin{array}{ll}
\sum_e 2^{-(e+1)} f_e & \mbox{\ if $\nu = \emptyset$;}\\
2^{-(\nu(0)+1)}f_{\nu(0)} \cdot \chi_{I(\nu)} & \mbox{\ otherwise.}\\
\end{array}
\right.
\]
Let $h$ be a computable surjection of $\N$ onto $S$, and let $\mathcal{B}^\# = (\mathcal{B}, \phi h)$.  

We divide the verification of our construction into the following lemmas.  Let 
$U = \sum_{e \in \Fin} \iota_e(\mathcal{B}_e)$, and let $V = \sum_{e \not \in \Fin} \iota_e(\mathcal{B}_e)$.  

\begin{lemma}\label{lm:isom.Lp}
$\mathcal{B}$ is isometrically isomorphic to $\ell^p \oplus L^p[0,1]$.  
\end{lemma}

\begin{proof}
Note that $\mathcal{B} = U + V$.
	If $e \in \Fin$, then $\mathcal{B}_e$ is a finite-dimensional $L^p$ space.  
	So, $U$ is isometrically isomorphic to $\ell^p$.  
	If $e \not \in \Fin$, then $\mathcal{B}_e = L^p[0,1]$.  
	So, $V$ is the $L^p$-sum of $L^p[0,1]$ with itself $\aleph_0$ times.
	However, this is the same thing as $L^p(\Omega)$ where $\Omega$ is the 
	product of Lebesgue measure on $[0,1]$ with itself $\aleph_0$ times.  
	As discussed in the introduction, this implies that $V$ is isometrically isomorphic to 
	$L^p[0,1]$.  
\end{proof}

It follows from the construction that $\phi$ is a disintegration of $\mathcal{B}$.

\begin{lemma}\label{cl:norm.phi.comp}
$\nu \mapsto \norm{\phi(\nu)}_{\mathcal{B}}$ is computable.
\end{lemma}

\begin{proof}
By construction, $\norm{\phi((e))}_{\mathcal{B}} = 2^{-(e+1)}$ for each $e$.  Thus, since $\phi$ is summative, 
$\phi(\emptyset) = 1$.  If $\nu'$ is a child of $\nu$ in $S$, then by construction 
$\norm{\phi(\nu')}_{\mathcal{B}}^p = \frac{1}{2} \norm{\phi(\nu)}_{\mathcal{B}}^p$.  It follows that 
$\nu \mapsto \norm{\phi(\nu)}_{\mathcal{B}}$ is computable.
\end{proof}

It now follows from Lemma \ref{lm:disintsArePresentations} that $\mathcal{B}^\#$ is a computable presentation of $\mathcal{B}$.   

\begin{lemma}\label{lm:decomp}
If $T$ is an isometric isomorphism of $\ell^p \oplus L^p[0,1]$ onto $\mathcal{B}$, then 
	$T[\ell^p \oplus \{\zerovec\}] =U$ and $T[\{\zerovec\} \oplus L^p[0,1]] = V$.
\end{lemma}

\begin{proof}
Suppose $T$ is an isometric isomorphism of $\ell^p \oplus L^p[0,1]$ onto $\mathcal{B}$.  
Let $U' = T[\ell^p \oplus \{\zerovec\}]$, and let $V' = T[\{\zerovec\} \oplus L^p[0,1]]$.  
Thus, $\mathcal{B}$ is the internal direct sum of $U'$ and $V'$.  

Suppose $j \in \N$, and let $T((e_j, \zerovec)) = f + g$ where $f \in U$ and $g \in V$.  
Since $T((e_j, \zerovec))$ is an atom of $\mathcal{B}$, and since there are no atoms in $V$, it follows that 
$g = 0$ and so $T((e_j, \zerovec)) \in U$.  We can then conclude that $U' \subseteq U$.  Conversely, 
suppose $e \in \Fin$ and $h = \iota_e(e_j)$.  Then, $T^{-1}(h)$ is an atom of $\ell^p \oplus L^p[0,1]$ and so $T^{-1}(h) \in \ell^p \oplus \{\zerovec\}$.  It follows that $\mathcal{B}_e \subseteq U'$ and so 
$U \subseteq U'$.  

Since $\mathcal{B}$ is the internal direct sum of $U$ and $V$, it now follows that $V = V'$.
\end{proof}

Let $P = P_V$.   

\begin{lemma}\label{lm:comp.P}
If $P$ is an $X$-computable map from $\mathcal{B}^\#$ into $\mathcal{B}^\#$, then $X$ computes $\Fin$.
\end{lemma}

\begin{proof}
	Suppose $X$ computes $P$ from $\mathcal{B}^\#$ into $\mathcal{B}^\#$.  If $\nu=(e)$, note that
	\[
	P(\phi(\nu)) = \left\{\begin{array}{ll}
	2^{-(e+1)}f_e & \mbox{\ if $e\not \in \Fin$;}\\
	\zerovec & \mbox{\ otherwise.}\\
	\end{array}
	\right.
	\]
	and
	\[
	\norm{P(\phi(\nu))}^p_\B = \left\{\begin{array}{ll}
	2^{-(e+1)} & \mbox{\ if $e \not \in \Fin$;}\\
	0 & \mbox{\ otherwise.}\\
	\end{array}
	\right.
	\]
	Given $e \in \N$, we can compute with oracle $X$ a rational number $q$ so that 
	$| \norm{P(\phi((e)))}_{\mathcal{B}}^p - q| < 2^{-(e+3)}$.  
	If $|q| < 2^{-(e+2)}$, then $\norm{P(\phi((e)))}_{\mathcal{B}}^p < 2^{-(e+1)}$ and so $e \in \Fin$.  Otherwise, $\norm{P(\phi(\nu))}_{\mathcal{B}}^p \neq 0$ and so $e \not \in \Fin$.
\end{proof}

Theorem \ref{thm:lpLp01.lower} now follows from Proposition \ref{prop:proj.comp}.  Note that we have also demonstrated the sharpness of the bounds in Theorem \ref{thm:proj.Lp.comp}.  

\section{Conclusion}\label{sec:conclusion}

Suppose $p \geq 1$ is a computable real with $p \neq 2$.  We have now classified the computably categorical $L^p$ spaces and determined the degrees of categoricity of those that are not computably categorical.  Our results relate the degree of categoricity of an $L^p$ space to the structure of the underlying measure space.  We have also determined the complexity of the natural projection operators on these spaces as well as their relationship to the degrees of categoricity.  In addition, we have provided the first example of a $\emptyset''$-categorical Banach space that is not $\emptyset'$-categorical.  This result leads to the following.

\begin{question}
If $n \in \N$ and $n \geq 2$, is there is a $\emptyset^{(n+1)}$-categorical Banach space that is not $\emptyset^{(n)}$-categorical?
\end{question}

We note that Melnikov and Nies have shown that each compact computable metric space is $\emptyset''$-categorical and that there is a compact computable Polish space that is not $\emptyset'$-categorical \cite{Melnikov.Nies.2013}.  

 We have shown that the degrees of $\emptyset$, $\emptyset'$, and $\emptyset''$ are degrees of categoricity of Banach spaces.  These results lead to the following.
  
\begin{question}
Is every hyperarithmetical degree the degree of categoricity of a Banach space?
\end{question}

\begin{question}
Is there a Banach space that does not have a degree of categoricity?
\end{question}

\section*{Acknowledgement}

We thank Diego Rojas for proofreading and several valuable suggestions.  


\begin{thebibliography}{10}

\bibitem{Anderson.Csima.2016}
B.~Anderson and B.~Csima, \emph{Degrees that are not degrees of categoricity},
  Notre Dame Journal of Formal Logic (2016).

\bibitem{Cembranos.Mendoza.1997}
Pilar Cembranos and Jos{{\'e}} Mendoza, \emph{Banach spaces of vector-valued
  functions}, Lecture Notes in Mathematics, vol. 1676, Springer-Verlag, Berlin,
  1997. \MR{1489231}

\bibitem{Clanin.McNicholl.Stull.2019}
Joe Clanin, Timothy~H. McNicholl, and Don~M. Stull, \emph{Analytic computable
  structure theory and {$L^p$} spaces}, Fund. Math. \textbf{244} (2019), no.~3,
  255--285.

\bibitem{Csima.Franklin.Shore.2013}
Barbara~F. Csima, Johanna N.~Y. Franklin, and Richard~A. Shore, \emph{Degrees
  of categoricity and the hyperarithmetic hierarchy}, Notre Dame J. Form. Log.
  \textbf{54} (2013), no.~2, 215--231.

\bibitem{Fokina.Kalimullin.Miller.2010}
Ekaterina~B. Fokina, Iskander Kalimullin, and Russell Miller, \emph{Degrees of
  categoricity of computable structures}, Arch. Math. Logic \textbf{49} (2010),
  no.~1, 51--67.

\bibitem{McNicholl.2015}
Timothy~H. McNicholl, \emph{A note on the computable categoricity of {$\ell^p$}
  spaces}, Evolving computability, Lecture Notes in Comput. Sci., vol. 9136,
  Springer, Cham, 2015, pp.~268--275.

\bibitem{McNicholl.2017}
Timothy~H. McNicholl, \emph{Computable copies of $\ell^p$}, Computability
  \textbf{6} (2017), no.~4, 391 -- 408.

\bibitem{Melnikov.2013}
Alexander~G. Melnikov, \emph{Computably isometric spaces}, J. Symbolic Logic
  \textbf{78} (2013), no.~4, 1055--1085.

\bibitem{Melnikov.Nies.2013}
Alexander~G. Melnikov and Andr{\'e} Nies, \emph{The classification problem for
  compact computable metric spaces}, The nature of computation, Lecture Notes
  in Comput. Sci., vol. 7921, Springer, Heidelberg, 2013, pp.~320--328.

\bibitem{Pour-El.Richards.1989}
Marian~B. Pour-El and J.~Ian Richards, \emph{Computability in analysis and
  physics}, Perspectives in Mathematical Logic, Springer-Verlag, Berlin, 1989.

\end{thebibliography}
\def\cprime{$'$}
\providecommand{\bysame}{\leavevmode\hbox to3em{\hrulefill}\thinspace}
\providecommand{\MR}{\relax\ifhmode\unskip\space\fi MR }
\providecommand{\MRhref}[2]{%
  \href{http://www.ams.org/mathscinet-getitem?mr=#1}{#2}
}
\providecommand{\href}[2]{#2}

\end{document}